\documentclass[12pt,letterpaper]{article}
\usepackage{ifthen,latexsym,amssymb,amsmath,amsthm}
\usepackage{graphicx}
\usepackage[hang]{caption}

\newdimen\unit\newdimen\psep\newcount\nd\newcount\ndx\newbox\dotb\newbox\ptbox
\newdimen\dx\newdimen\dy\newdimen\dxx\newdimen\dyy\newdimen\hgt
\newdimen\xoff\newdimen\yoff
\newcommand\clap[1]{\hbox to 0pt{\hss{#1}\hss}}
\newcommand\vdisk[1]{{\font\dotf=cmr10 scaled #1\dotf.}}
\newcommand\varline[2]{\setbox\dotb\hbox{\vdisk{#1}}\xoff=-.5\wd\dotb
\wd\dotb=0pt\yoff=-.5\ht\dotb\psep=#2\ht\dotb}
\newcommand\varpt[1]{\setbox\ptbox\clap{\vdisk{#1}}\setbox\ptbox
\hbox{\raise-.5\ht\ptbox\box\ptbox}}
\newcommand\cpt{\copy\ptbox}
\newcommand\point[3]{\rlap{\kern#1\unit\raise#2\unit\hbox{#3}}}
\newcommand\setnd[4]{\dx=#3\unit\advance\dx-#1\unit\divide\dx by\psep
\dy=#4\unit\advance\dy-#2\unit\divide\dy by\psep \multiply\dx
by\dx\multiply\dy by\dy\advance\dx\dy\nd=1\advance\dx-1sp
\loop\ifnum\dx>0\advance\dx-\nd sp\advance\nd1\advance\dx-\nd
sp\repeat}
\newcommand\dl[4]{{\setnd{#1}{#2}{#3}{#4}\dline{#1}{#2}{#3}{#4}\nd}}
\newcommand\dline[5]{{\nd=#5\hgt=#2\unit\dx=#3\unit\advance\dx-#1\unit
\divide\dx by\nd\dy=#4\unit\advance\dy-#2\unit\divide\dy by\nd
\advance\hgt\yoff\rlap{\kern#1\unit\kern\xoff\loop\ifnum\nd>1\advance\nd-1
\advance\hgt\dy\kern\dx\raise\hgt\copy\dotb\repeat}}}
\newcommand\ellipse[4]{\qellip{#1}{#2}{#3}{#4}\qellip{#1}{#2}{#3}{-#4}%
\qellip{#1}{#2}{-#3}{#4}\qellip{#1}{#2}{-#3}{-#4}}
\newcommand\qellip[4]{{\setnd{0}{0}{#3}{#4}\dx=\unit\dy=0pt\raise\yoff\rlap{%
\kern#1\unit\kern\xoff\raise#2\unit\hbox{\loop\ifnum\dx>0\rlap{\kern#3\dx
\raise#4\dy\copy\dotb}\hgt=\dx\divide\hgt
by\nd\advance\dy\hgt\hgt=\dy \divide\hgt
by\nd\advance\dx-\hgt\repeat\rlap{\raise#4\dy\copy\dotb}}}}}
\newcommand\bez[6]{{\setnd{#1}{#2}{#3}{#4}\ndx=\nd\setnd{#3}{#4}{#5}{#6}
\ifnum\ndx>\nd\nd=\ndx\fi\dx=#3\unit\advance\dx-#1\unit\dy=#4\unit
\advance\dy-#2\unit\dxx=#5\unit\advance\dxx-#1\unit\dyy=#6\unit\advance
\dyy-#2\unit\advance\dxx-2\dx\advance\dyy-2\dy\divide\dxx
by\nd\divide\dyy
by\nd\advance\dx.25\dxx\advance\dy.25\dyy\divide\dx
by\nd\divide\dy by\nd \multiply\nd
by2\dx=100\dx\dy=100\dy\dxx=100\dxx\dyy=100\dyy\divide\dxx by\nd
\divide\dyy
by\nd\hgt=#2\unit\raise\yoff\rlap{\kern#1\unit\kern\xoff
\raise\hgt\copy\dotb\loop\ifnum\nd>0\advance\nd-1\advance\hgt0.01\dy
\kern0.01\dx\raise\hgt\copy\dotb\advance\dx\dxx\advance\dy\dyy\repeat}}}
\newcommand\ptu[3]{\point{#1}{#2}{\cpt\raise1ex\clap{$\scriptstyle{#3}$}}}
\newcommand\ptd[3]{\point{#1}{#2}{\cpt\raise-1.8ex\clap{$\scriptstyle{#3}$}}}
\newcommand\ptr[3]{\point{#1}{#2}{\cpt\raise-.4ex\rlap{$\ \scriptstyle{#3}$}}}
\newcommand\ptl[3]{\point{#1}{#2}{\cpt\raise-.4ex\llap{$\scriptstyle{#3}\ $}}}
\newcommand\ptlu[3]{\point{#1}{#2}{\raise.8ex\clap{$\scriptstyle{#3}$}}}
\newcommand\ptld[3]{\point{#1}{#2}{\raise-1.6ex\clap{$\scriptstyle{#3}$}}}
\newcommand\ptlr[3]{\point{#1}{#2}{\raise-.4ex\rlap{$\,\scriptstyle{#3}$}}}
\newcommand\ptll[3]{\point{#1}{#2}{\raise-.4ex\llap{$\scriptstyle{#3}\,$}}}
\newcommand\pt[2]{\point{#1}{#2}{\cpt}}

\newcommand\thnline{\varline{400}{.6}}
\newcommand\dotline{\varline{1000}{4}}
\varpt{2500}\thnline\unit=2em

\title{Monochromatic Clique Decompositions of Graphs}
\date{}

\author{Henry Liu\\
\small{Centro de Matem\'atica e Aplica\c c\~oes} \\
\small{Faculdade de Ci\^encias e Tecnologia, Universidade Nova de Lisboa}\\
\small{Campus de Caparica, 2829-516 Caparica, Portugal}\\
\small{\texttt{h.liu@fct.unl.pt}}
\and Oleg Pikhurko\\
\small{Mathematics Institute and DIMAP}\\
\small{University of Warwick}\\
\small{Coventry CV4 7AL, United Kingdom}\\
\small{\texttt{http://homepages.warwick.ac.uk/staff/O.Pikhurko}}
\and Teresa Sousa\\
\small{Departamento de Matem\'atica and Centro de Matem\'atica e Aplica\c c\~oes} \\ 
\small{Faculdade de Ci\^encias e Tecnologia, Universidade Nova de Lisboa}\\
\small{Campus de Caparica, 2829-516 Caparica, Portugal}\\
\small{\texttt{tmjs@fct.unl.pt}}}

\begin{document}
\newtheorem{df}{Definition}[section]
\newtheorem{thm}[df]{Theorem}
\newtheorem{lm}[df]{Lemma}
\newtheorem{prop}[df]{Proposition}
\newtheorem{conj}[df]{Conjecture}
\newtheorem{cor}[df]{Corollary}
\newtheorem{clm}[df]{Claim}
\newtheorem*{clm1}{Claim 1}
\newtheorem*{clm2}{Claim 2}
\renewcommand{\thedf}{\thesection.\arabic{df}}
\renewcommand{\thethm}{\thesection.\arabic{thm}}
\renewcommand{\theprop}{\thesection.\arabic{prop}}
\renewcommand{\thelm}{\thesection.\arabic{lm}}
\newcommand{\ex}{\mathrm{ex}}
\newcommand{\pf}{\textit{Proof: }}
\numberwithin{equation}{section} 

\def\dcup{\,\dot\cup\,}
\def\eps{\varepsilon}

\maketitle
\renewcommand{\baselinestretch}{1.1}

\begin{abstract}
Let $G$ be a graph whose edges are coloured with $k$ colours, and $\mathcal H=(H_1,\dots , H_k)$ be a $k$-tuple of graphs. A \emph{monochromatic $\mathcal H$-decomposition} of $G$ is a partition of the edge set of $G$ such that each part is either a single edge or forms a monochromatic copy of $H_i$ in colour $i$, for some $1\le i\le k$. Let $\phi_{k}(n,\mathcal H)$ be the smallest number $\phi$, such that, for every
order-$n$ graph and every $k$-edge-colouring, there is a monochromatic $\mathcal H$-decomposition with at most $\phi$ elements. Extending the previous results of Liu and Sousa [``Monochromatic $K_r$-decompositions of graphs", \emph{Journal of Graph Theory}, 76:89--100, 2014], we solve this problem 
when each graph in $\mathcal H$ is a clique and $n\ge n_0(\mathcal H)$ is sufficiently large.\\

\end{abstract}

\section{Introduction}\label{intro}

All graphs in this paper are finite, undirected and simple. For standard graph-theoretic terminology the reader is referred to \cite{BB98}.

Given two graphs $G$ and $H$, an \emph{$H$-decomposition} of $G$ is a partition of the edge set of $G$ such that each part is either a single edge or forms a subgraph isomorphic to $H$. Let $\phi(G,H)$ be the smallest possible number of parts in an $H$-decomposition of $G$. It is easy to see that, if $H$ is non-empty, we have $\phi(G,H)=e(G)-\nu_H(G)(e(H)-1)$, where $\nu_H(G)$ is the maximum number of pairwise edge-disjoint copies of $H$ that can be packed into $G$. Dor and Tarsi~\cite{DT97} showed that if $H$ has a component with at least 3 edges then it is NP-complete to determine if a graph $G$ admits a partition into copies of $H$. Thus, it is NP-hard to compute the function $\phi(G,H)$ for such $H$. Nonetheless, many exact results 
were proved about the extremal function 
 $$
 \phi(n,H)=\max\{\phi(G,H)\mid v(G)=n\},
 $$
which is the smallest number such that any graph $G$ of order $n$ admits an $H$-decomposition with at most $\phi(n,H)$ elements.

This function was first studied, in 1966, by Erd\H os, Goodman and P\'osa~\cite{EGP66}, who proved that $\phi(n, K_3)=t_2(n)$, where $K_s$ denotes the complete graph  (clique) of order $s$, and $t_{r-1}(n)$ denotes the number of edges in the \emph{Tur\'an graph} $T_{r-1}(n)$, which is the unique $(r-1)$-partite graph on $n$ vertices that has the maximum number of edges. A decade later,  Bollob\'as~\cite{BB76} proved that $\phi(n,K_r)=t_{r-1}(n)$, for all $n\ge r\ge 3$.

Recently Pikhurko and Sousa \cite{PS07} studied $\phi(n,H)$ for arbitrary graphs $H$. Their result is the following. 

\begin{thm}\label{PSthm}\textup{\cite{PS07}}
Let $H$ be any fixed graph of chromatic number $r\ge 3$. Then,
$$\phi(n,H)=t_{r-1}(n)+o(n^2).$$
\end{thm}

Let $\ex(n,H)$ denote the maximum number of edges in a graph on $n$ vertices not containing $H$ as a subgraph. The result of Tur\'an \cite{PT41} states that $T_{r-1}(n)$ is the unique extremal graph for $\ex(n,K_r)$. The function $\ex(n,H)$ is usually called the \emph{Tur\'an function} for $H$. Pikhurko and Sousa \cite{PS07} also made the following conjecture. 

\begin{conj}\emph{\cite{PS07}}\label{PSconj}
For any graph $H$ of chromatic number $r\ge 3$, there exists $n_0=n_0(H)$ such that $\phi(n,H)=\ex(n,H)$ for all $n\ge n_0$.
\end{conj}

A graph $H$ is \emph{edge-critical} if there exists an edge $e\in E(H)$ such that $\chi(H)>\chi(H-e)$, where $\chi(H)$ denotes the \emph{chromatic number} of $H$. For $r\geq 4$, a \emph{clique-extension of order $r$} is a connected graph that consists of a $K_{r-1}$ plus another vertex, say $v$, adjacent to at most $r-2$ vertices of $K_{r-1}$. Conjecture \ref{PSconj} has been verified by Sousa for some edge-critical graphs, namely, clique-extensions of order $r\geq 4$ ($n\geq r$) \cite{TS11}  and the cycles of length 5 ($n\geq 6$) and 7 ($n\geq10$)~\cite{TS05,TS:C7}. Later, \"Ozkahya and Person~\cite{OP12} verified the conjecture for all edge-critical graphs with chromatic number $r\geq 3$.   Their result is the following. 

\begin{thm}\emph{\cite{OP12}}\label{OPthm}
For any edge-critical graph $H$ with chromatic number $r\geq 3$, there exists $n_0=n_0(H)$ such that $\phi(n,H)=\ex(n,H)$, for all $n\geq n_0$. Moreover, the only graph attaining $\ex(n,H)$ is the Tur\'{a}n graph $T_{r-1}(n)$.
\end{thm}

Recently, as an extension of \"Ozkahya and Person's work (and as further evidence supporting Conjecture \ref{PSconj}), Allen, B\"ottcher, and Person \cite{ABP12+} improved the error term obtained by Pikhurko and Sousa in Theorem \ref{PSthm}. In fact,  they proved that the error term $o(n^2)$ can be replaced by $O(n^{2-\alpha})$ for some $\alpha>0$. Furthermore,  they also showed that this error term has the correct order of magnitude. Their result is indeed an extension of Theorem \ref{OPthm} since the error term $O(n^{2-\alpha})$ that they obtained vanishes for every edge-critical graph $H$.  

Motivated by the recent work about $H$-decompositions of graphs, a natural problem to consider is the Ramsey (or coloured) version of this problem. More precisely, let $G$ be a graph on $n$ vertices whose edges are coloured with $k$ colours, for some $k\geq 2$ and let $\mathcal H=(H_1,\dots , H_k)$ be a $k$-tuple of fixed graphs, where repetition is allowed. A \emph{monochromatic $\mathcal H$-decomposition} of $G$ is a partition of its edge set  such that each part is either a single edge, or forms a monochromatic copy of $H_i$ in colour $i$, for some $1\le i\le k$. Let $\phi_{k}(G,\mathcal H)$ be the smallest number, such that, for any $k$-edge-colouring of $G$, there exists a monochromatic $\mathcal H$-decomposition of $G$ with at most $\phi_{k}(G,\mathcal H)$ elements. Our goal is to study the function
\begin{equation*}
\phi_{k}(n,\mathcal H) = \max \{\phi_{k}(G,\mathcal H) \mid v(G)=n  \},
\end{equation*}
which is the smallest number $\phi$ such that, any $k$-edge-coloured graph of order $n$ admits a monochromatic $\mathcal H$-decomposition with at most $\phi$ elements. In the case when $H_i\cong H$ for every $1\le i\le k$, we simply write $\phi_{k}(G,H)=\phi_{k}(G,\mathcal H)$ and $\phi_{k}(n,H)=\phi_{k}(n,\mathcal H)$.

The function $\phi_{k}(n,K_r)$, for $k\ge 2$ and $r\ge 3$, has been studied by Liu and Sousa \cite{LS13}, who obtained results involving the Ramsey numbers and the Tur\'an numbers.  Recall that for $k\ge 2$ and integers $r_1,\dots,r_k\ge 3$, the \emph{Ramsey number for $K_{r_1},\dots , K_{r_k}$}, denoted by $R(r_1,\dots , r_k)$, is the smallest value of $s$, such that,  for every $k$-edge-colouring of $K_s$, there exists a monochromatic $K_{r_i}$ in colour $i$, for some $1\le i\le k$. For the case when $r_1=\cdots=r_k=r$, for some $r\ge 3$, we simply write $R_k(r)=R(r_1,\dots , r_k)$. 
Since $R(r_1,\dots , r_k)$ does not change under any permutation of $r_1,\dots , r_k$, without loss of generality, we assume throughout that $3\le r_1\le\cdots\le r_k$. The Ramsey numbers are notoriously difficult to calculate, even though, it is known that their values  are finite \cite{FR30}. To this date, the values of $R(3,r_2)$ have been determined exactly only for $3\le r_2\le 9$, and these are shown in the following table~\cite{R96}.
\[
\begin{array}{|c||c|c|c|c|c|c|c|c|c|c|c|c|c|}
\hline
r_2 & 3 & 4 & 5 & 6 & 7 & 8 & 9\\
\hline
R(3,r_2) & 6 & 9 & 14 & 18 & 23 & 28 & 36\\
\hline
\end{array}
\]

The remaining Ramsey numbers that are known exactly are $R(4,4)=18$, $R(4,5)=25$, and $R(3,3,3)=17$. The gap between the lower bound and the upper bound for other Ramsey numbers is generally quite large. 

For the case $R(3,3)=6$, it is easy to see that the only $2$-edge-colouring of $K_5$ not containing a monochromatic $K_3$  is the one where each colour induces a cycle of length $5$. From this $2$-edge-colouring, observe that we may take a `blow-up' to obtain a $2$-edge-colouring of the Tur\'an graph $T_5(n)$, and easily deduce that $\phi_2(n,K_3)\ge t_5(n)$. See Figure~\ref{blow-up of K_5}. \\[1ex] \label{blow-up of K_5}

\[ \unit = 1cm
\thnline 
\dl{-2}{1}{-4}{1}\dl{-4}{1}{-4.62}{2.9}\dl{-2}{1}{-1.38}{2.9}\dl{-1.38}{2.9}{-3}{4.08}\dl{-4.62}{2.9}{-3}{4.08}
\dotline
\dl{-2}{1}{-4.62}{2.9}\dl{-4.62}{2.9}{-1.38}{2.9}\dl{-1.38}{2.9}{-4}{1}\dl{-4}{1}{-3}{4.08}\dl{-2}{1}{-3}{4.08}
\pt{-2}{1}\pt{-4}{1}\pt{-4.62}{2.9}\pt{-1.38}{2.9}\pt{-3}{4.08}
%
\thnline 
\dl{1.9}{1}{3.1}{1}\dl{1.9}{0.92}{3.1}{0.92}\dl{1.9}{1.08}{3.1}{1.08}
\dl{3.61}{1.37}{4}{2.53}\dl{3.54}{1.4}{3.93}{2.56}\dl{3.68}{1.34}{4.07}{2.5}
\dl{1.39}{1.37}{1}{2.53}\dl{1.46}{1.4}{1.07}{2.56}\dl{1.32}{1.34}{0.93}{2.5}
\dl{1.2}{3.13}{2.18}{3.85}\dl{1.24}{3.07}{2.22}{3.79}\dl{1.16}{3.19}{2.14}{3.91}
\dl{3.8}{3.13}{2.82}{3.85}\dl{3.76}{3.07}{2.78}{3.79}\dl{3.84}{3.19}{2.86}{3.91}
\dotline
\dl{1.73}{1.17}{3.89}{2.73}\dl{1.68}{1.24}{3.84}{2.8}\dl{1.78}{1.1}{3.94}{2.66}
\dl{3.27}{1.17}{1.11}{2.73}\dl{3.32}{1.24}{1.16}{2.8}\dl{3.22}{1.1}{1.06}{2.66}
\dl{3.85}{2.9}{1.15}{2.9}\dl{3.85}{2.82}{1.15}{2.82}\dl{3.85}{2.98}{1.15}{2.98}
\dl{1.59}{1.25}{2.41}{3.83}\dl{1.51}{1.28}{2.33}{3.86}\dl{1.67}{1.22}{2.49}{3.8}
\dl{3.41}{1.25}{2.59}{3.83}\dl{3.49}{1.28}{2.67}{3.86}\dl{3.33}{1.22}{2.51}{3.8}
\thnline
\ellipse{1.5}{1}{0.5}{0.5}\ellipse{3.5}{1}{0.5}{0.5}\ellipse{4.12}{2.9}{0.5}{0.5}\ellipse{0.88}{2.9}{0.5}{0.5}\ellipse{2.5}{4.08}{0.5}{0.5}
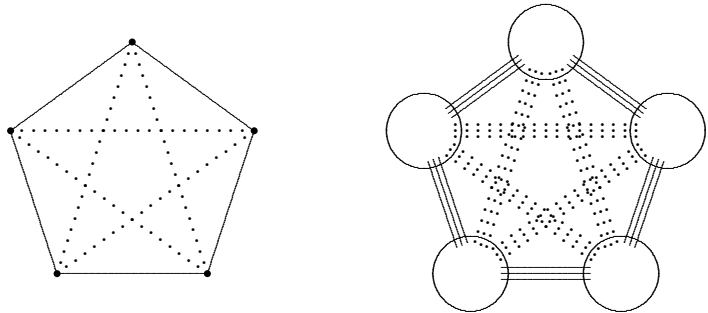
\point{-4.9}{-0.5}{Figure~\ref{blow-up of K_5}. The $2$-edge-colouring of $K_5$, and its blow-up}
\]\\[1ex]

This example was the motivation for Liu and Sousa \cite{LS13} to study $K_r$-monochromatic decompositions of graphs, for $r\geq 3$ and $k\geq 2$. They have recently proved the following result.

\begin{thm}\emph{\cite{LS13}}\label{LSthm}
\begin{enumerate}
\item[(a)] $\phi_k(n,K_3) = t_{R_k(3)-1}(n)+o(n^2)$;
\item[(b)]  $\phi_k(n,K_3) = t_{R_k(3)-1}(n)$ for $k=2,3$ and $n$ sufficiently large; 
\item[(c)] $\phi_k(n,K_r) = t_{R_k(r)-1}(n)$, for $k\geq 2$, $r\geq 4$ and $n$ sufficiently large. 
\end{enumerate}

Moreover, the only graph attaining $\phi_k(n,K_r)$ in cases (b) and (c) is the Tur\'an graph $T_{R_k(r)-1}(n)$.
\end{thm}

They also made the following conjecture.
\begin{conj} \textup{\cite{LS13}}\label{LSconj}
Let $k\ge 4$. Then $\phi_k(n,K_3) = t_{R_k(3)-1}(n)$ for $n\ge R_k(3)$.
\end{conj}

Here, we will study an extension of the monochromatic $K_r$-decomposition problem when the clique $K_r$ is replaced by a fixed $k$-tuple of cliques $\mathcal C=(K_{r_1},\dots,K_{r_k})$. Our main result, stated in Theorem \ref{Krexact}, is clearly an  extension of Theorem \ref{LSthm}. Also, it verifies Conjecture \ref{LSconj} for sufficiently large $n$. 

\begin{thm}\label{Krexact}
Let $k\geq 2$, $3\le r_1\le\cdots\le r_k$, and $R=R(r_1,\dots,r_k)$. Let $\mathcal C=(K_{r_1},\dots,K_{r_k})$. Then, there is an $n_0=n_0(r_1,\dots,r_k)$ such that, for all $n\geq n_0$, we have 
\begin{equation*}
\phi_k(n,\mathcal C) = t_{R-1}(n).\label{Krthmeq}
\end{equation*}

Moreover, the only order-$n$ graph attaining $\phi_k(n,\mathcal C) $ is the Tur\'{a}n graph $T_{R-1}(n)$ (with a $k$-edge-colouring that does not contain a colour-i copy of $K_{r_i}$ for any $1\le i\le k$).
\end{thm}

The upper bound of Theorem~\ref{Krexact} is proved in Section~\ref{Krsect}. The lower bound follows easily by the definition of the Ramsey number. Indeed, take a $k$-edge-colouring $f'$ of the complete graph $K_{R-1}$ without a  monochromatic $K_{r_i}$ in colour $i$, for all $1\le i\le k$. Note that $f'$ exists by definition of the Ramsey number $R=R(r_1,\dots,r_k)$.
Let $u_1,\dots, u_{R-1}$ be the vertices of the $K_{R-1}$. Now, consider the Tur\'an graph $T_{R-1}(n)$ with a $k$-edge-colouring $f$ which is a `blow-up' of $f'$. That is, if $T_{R-1}(n)$ has partition classes $V_1,\dots , V_{R-1}$, then for $v\in V_j$ and $w\in V_\ell$ with $j\neq\ell$, we define $f(vw)=f'(u_ju_\ell)$. Then, $T_{R-1}(n)$ with this $k$-edge-colouring has no monochromatic $K_{r_i}$ in colour $i$, for every $1\le i\le k$. Therefore, $\phi_k(n,\mathcal C)\ge \phi_k(T_{R-1}(n), \mathcal C) =t_{R-1}(n)$ and the lower bound in Theorem~\ref{Krexact} follows.

In particular, when all the cliques in $\mathcal  C$ are equal, Theorem~\ref{Krexact} completes the results obtained previously by Liu and Sousa in Theorem~\ref{LSthm}. In fact, we get the following direct corollary from Theorem~\ref{Krexact}. 

\begin{cor} Let $k\geq 2$, $r\ge 3$ and $n$ be sufficiently large. Then, 
$$\phi_k(n,K_r) =t_{R_k(r)-1}(n).$$

Moreover, the only order-$n$ graph attaining $\phi_k(n,K_r)$  is the Tur\'an graph $T_{R_k(r)-1}(n)$ (with a $k$-edge-colouring that does not contain a monochromatic copy of $K_{r}$).
\end{cor}

\section{Proof of Theorem \ref{Krexact}}\label{Krsect}

In this section we will prove the upper bound in Theorem \ref{Krexact}. Before presenting the proof  we need to introduce the tools. Throughout this section, let $k\geq 2$, $3\le r_1\le\cdots\le r_k$ be an increasing sequence of integers, $R=R(r_1,\dots,r_k)$ be the Ramsey number for $K_{r_1},\dots,K_{r_k}$, and $\mathcal C=(K_{r_1},\dots,K_{r_k})$ be a fixed $k$-tuple of cliques. 

We first recall the following stability theorem of Erd\H{o}s and Simonovits \cite{E67,S68}.

\begin{thm}[Stability Theorem \cite{E67,S68}]\label{ESthm}
Let $r\ge 3$, and $G$ be a graph on $n$ vertices with $e(G)\ge t_{r-1}(n)+o(n^2)$ and not containing $K_r$ as a subgraph. Then, there exists an $(r-1)$-partite graph $G'$ on $n$ vertices with partition classes $V_1,\dots,V_{r-1}$, where $|V_i|=\frac{n}{r-1}+o(n)$ for $1\le i\le r-1$, that can be obtained from $G$ by adding and subtracting $o(n^2)$ edges. 
\end{thm}

Next, we recall the following result of Gy\H{o}ri \cite{EG88,EG91} about the existence of edge-disjoint copies of $K_r$ in graphs on $n$ vertices with more than $t_{r-1}(n)$ edges.

\begin{thm}\label{EGthm}\textup{\cite{EG88,EG91}}
For every $r\ge 3$ there is $C$ such that every graph $G$ with $n\ge C$ vertices and $e(G)=t_{r-1}(n)+m$ edges, where $m\le {n\choose 2}/C$, contains at least $m-Cm^2/n^2$ edge-disjoint copies of $K_r$.
\end{thm}

Now, we will consider coverings and packings of cliques in graphs. Let $r\ge 3$ and $G$ be a graph.
Let $\mathcal K$ be the set of all $K_r$-subgraphs of $G$. A \emph{$K_r$-cover} is a set of edges of $G$ meeting all elements in $\mathcal K$, that is, the removal of a $K_r$-cover results in a $K_r$-free graph. A \emph{$K_r$-packing} in $G$ is a set of pairwise edge-disjoint copies of $K_r$. The \emph{$K_r$-covering number} of $G$, denoted by $\tau_r(G)$, is the minimum size of a $K_r$-cover of $G$, and the \emph{$K_r$-packing number} of $G$, denoted by $\nu_r(G)$, is the maximum size of a $K_r$-packing of $G$. 
Next, a \emph{fractional $K_r$-cover} of $G$ is a function $f:E(G)\rightarrow \mathbb{ R}_+$, such that $\sum_{e\in E(H)}f(e)\ge 1$ for every $H\in\mathcal K$, that is, for every copy of $K_r$ in $G$ the sum of the values of $f$ on its edges is at least $1$. A \emph{fractional $K_r$-packing} of $G$ is a function $p: \mathcal{K}\rightarrow \mathbb{ R}_+$ such that   $\sum_{H\in\mathcal K: e \in E(H)}p(H)\le 1$ for every $e\in E(G)$, that is, the total weight of $K_r$'s that cover any edge is at most 1. Here, $\mathbb{ R}_+$ denotes the set of non-negative real numbers. The \emph{fractional $K_r$-covering number} of $G$, denoted by $\tau_r^*(G)$, is the minimum of $\sum_{e\in E(G)} f(e)$ over all fractional $K_r$-covers $f$, and the \emph{fractional $K_r$-packing number} of $G$, denoted by $\nu_r^*(G)$, is the maximum of $\sum_{H\in \mathcal{K}} p(H)$ over all fractional $K_r$-packings $p$. 

One can easily observe that
\[
\nu_r(G) \leq \tau_r(G) \leq {r\choose 2}\nu_r(G).
\]

For $r=3$, we have $\tau_3(G) \leq 3\nu_3(G)$. A long-standing conjecture of Tuza \cite{ZT81} from 1981 states that this inequality can be improved as follows.

\begin{conj}\label{ZTconj}\textup{\cite{ZT81}}
For every graph $G$, we have $\tau_3(G) \leq 2\nu_3(G).$
\end{conj}

Conjecture \ref{ZTconj} remains open although many partial results have been proved.  By using the earlier results of Krivelevich \cite{MK95}, and Haxell and R\"{o}dl \cite{HR01}, Yuster \cite{Y12} proved the following theorem which will be crucial to the proof of Theorem \ref{Krexact}. In the case $r=3$, it is an asymptotic solution of Tuza's conjecture.

\begin{thm}\label{RYthm}\textup{\cite{Y12}}
Let $r\ge 3$ and $G$ be a graph on $n$ vertices. Then 
\begin{equation}
\tau_r(G)\le \Big\lfloor\frac{r^2}{4}\Big\rfloor\nu_r(G)+o(n^2).\label{yustereq}
\end{equation}
\end{thm}

We now prove the following lemma which states that a graph $G$ with $n$ vertices and at least $t_{R-1}(n)+\Omega(n^2)$ edges falls quite short of being optimal.

\begin{lm}\label{lm:new} 
For every $k\ge 2$ and $c_0>0$ there are $c_{1}>0$ and $n_0$ such that for every graph $G$ of order $n\ge n_0$ with at least $t_{R-1}(n)+c_0 n^2$ edges, we have $\phi_k(G,\mathcal{C})\le t_{R-1}(n)-c_{1}n^2$. 
\end{lm}

\begin{proof}
Suppose that the lemma is false, that is, there is $c_0>0$ such that for some increasing sequence
of $n$ there is a graph $G$ on $n$ vertices with $e(G)\ge t_{R-1}(n)+c_0 n^2$ and $\phi_k(G,\mathcal{C} )\ge t_{R-1}(n)+o(n^2)$. Fix a $k$-edge-colouring of $G$ and, for $1\leq i\leq k$, let $G_i$ be the subgraph of $G$ on $n$ vertices that contains all edges with colour $i$.  

Let $m=e(G)-t_{R-1}(n), $  and let  $s\in\{0,\dots,k\}$ be the maximum such that $$r_1=\dots=r_s=3.$$

Let us very briefly recall the argument from \cite{LS13} that shows $\phi_k(G,\mathcal{C})\le t_{R-1}(n)+o(n^2)$, adopted to our purposes. If we remove a $K_{r_i}$-cover from $G_i$ for every $1\le i\le k$, then we destroy
all copies of $K_R$ in $G$. By Tur\'an's theorem, at most $t_{R-1}(n)$ edges remain. Thus, 
 \begin{equation}\label{eq:tauG}
  \sum_{i=1}^k \tau_{r_i}(G_i)\ge m.
   \end{equation}
   
By Theorem \ref{RYthm}, if we decompose $G$ into a maximum $K_{r_i}$-packing in each
$G_i$ and the remaining edges, we obtain that 
 \begin{eqnarray}
  \phi_k(G,\mathcal{ C})&\le& e(G)-\sum_{i=1}^k \left({r_i\choose 2}-1\right)\nu_{r_i}(G_i)\nonumber\\
   &\le& t_{R-1}(n) +m -   
   \sum_{i=1}^k \frac{{r_i\choose 2}-1}{\lfloor r_i^2/4\rfloor}\, \tau_{r_i}(G_i) +o(n^2)\label{eq:main} \\
    &\le & t_{R-1}(n) +m-\sum_{i=1}^k \tau_{r_i}(G_i)-\frac14\sum_{i=s+1}^k \tau_{r_i}(G_i)+o(n^2)\ \le\ t_{R-1}(n) +o(n^2). \nonumber
   \end{eqnarray}
   The third inequality holds since   $({r\choose 2}-1)/\lfloor r^2/4\rfloor\ge 5/4$ for $r\ge 4$ and is equal to $1$ for $r=3$.

Let us derive a contradiction from this by looking at the properties of our hypothetical counterexample $G$. First, all inequalities that we saw have to be equalities within an additive term $o(n^2)$. In particular, the slack in (\ref{eq:tauG}) is $o(n^2)$, that is,
 \begin{equation}\label{eq:tauGEq}
  \sum_{i=1}^k \tau_{r_i}(G_i)= m+o(n^2).
   \end{equation}
   
Also, $\sum_{i=s+1}^k \tau_{r_i}(G_i)=o(n^2)$. In particular, we have that $s\ge 1$. To simplify the later calculations, let us re-define $G$ by removing a maximum $K_{r_i}$-packing from $G_i$ for each $i\ge s+1$. The new graph is still a counterexample to the lemma if we decrease $c_0$ slightly,
since the number of edges removed is at most $\sum_{i=s+1}^k {r_i\choose 2}\tau_{r_i}(G_i)=o(n^2)$.

Suppose that we remove, for each $i\le s$, an arbitrary (not necessarily minimum) $K_3$-cover $F_i$ from $G_i$ such that
 \begin{equation}\label{eq:sumFi}
  \sum_{i=1}^s |F_i|\le m+o(n^2).
   \end{equation}
   
Let $G'\subseteq G$ be the obtained $K_R$-free graph. (Recall that we assumed that $G_i$ is $K_{r_i}$-free for all $i\ge s+1$.) Let $G_i'\subseteq G_i$ be the colour classes of $G'$. We know by (\ref{eq:sumFi})
that $e(G')\ge t_{R-1}(n)+o(n^2)$. Since $G'$ is $K_R$-free, we conclude by the Stability Theorem (Theorem \ref{ESthm}) that there is a partition $V(G)=V(G')=V_1\,\dot\cup\,\dots\,\dot\cup\, V_{R-1}$ such that
\begin{equation}\label{eq:diff} 
 \forall\,  i \in \{1,\ldots, R-1\} , \quad  |V_i|=\frac n{R-1}+o(n) \qquad  \mbox{and}\qquad  |E(T)\setminus E(G')|=o(n^2),
\end{equation}
  where $T$ is the complete $(R-1)$-partite graph with parts $V_1,\dots,V_{R-1}$.
  
Next, we essentially expand the proof of (\ref{yustereq}) for $r=3$ and transform it into an algorithm
that produces $K_3$-coverings $F_i$ of $G_i$, with $1\le i\le s$, in such a way that  (\ref{eq:sumFi}) holds 
but (\ref{eq:diff}) is impossible whatever $V_1,\dots,V_{R-1}$ we take, giving the desired contradiction.

Let $H$ be an arbitrary graph of order $n$. By the LP duality, we have that 
 \begin{equation}\label{eq:duality}
  \tau_r^*(H)=\nu_r^*(H).
   \end{equation}
   
By the result of Haxell and R\"odl  \cite{HR01} we have that
 \begin{equation}\label{eq:HR}
  \nu_r^*(H)=\nu_r(H)+o(n^2).
   \end{equation}
   
   Krivelevich \cite{MK95} showed that
 \begin{equation}\label{eq:tauK3}
  \tau_3(H)\le 2\tau_3^*(H).
   \end{equation}
   
Thus, $\tau_3(H)\le 2\nu_3(H)+o(n^2)$ giving  (\ref{yustereq}) for $r=3$.  
  
The proof of   Krivelevich \cite{MK95} of (\ref{eq:tauK3}) is based on
the following result.

\begin{lm}\label{lm:K} Let $H$ be an arbitrary graph and $f:E(H)\rightarrow \mathbb{ R}_+$ be a minimum
fractional $K_3$-cover. Then $\tau_3(H)\le \frac32\, \tau_3^*(H)$ or there is
$xy\in E(H)$ with $f(xy)=0$ that belongs to at least one triangle of $H$.
\end{lm}

\begin{proof} If there is an edge $xy\in E(H)$ that does not belong to a triangle,
 then necessarily $f(xy)=0$
and $xy$ does not belong to any optimal fractional or integer $K_3$-cover. We can remove $xy$ from $E(H)$
without changing the validity of the lemma. Thus, we can assume that every edge of $H$ belongs to
a triangle. 

Suppose that $f(xy)>0$ for every edge $xy$ of $H$, for otherwise we are done. Take a maximum
fractional $K_3$-packing $p$. Recall that it is a function that assigns a weight $p(xyz)\in \mathbb{R}_+$ to each triangle $xyz$ of $H$ such that for every edge $xy$ the sum of weights over all $K_3$'s of $H$ containing $xy$ is at most 1, that is, 
  \begin{equation}\label{eq:xy}
  \sum_{z\in \Gamma(x)\cap \Gamma(y)} p(xyz)\le 1, 
   \end{equation}
where $\Gamma(v)$ denotes the set of neighbours of the  vertex $v$ in $H$. 

  This is the dual LP to the minimum fractional $K_3$-cover problem. 
By the complementary slackness condition (since $f$ and $p$ are optimal solutions), 
we have equality in (\ref{eq:xy}) for every $xy\in E(H)$.  This and the LP duality
imply that 
 $$\tau_3^*(H)=
  \nu_3^*(H)=\sum_{\mathrm{triangle\ }xyz} p(xyz)=\frac13 \sum_{xy\in E(H)} \sum_{z\in \Gamma(x)\cap \Gamma(y)} p(xyz) = \frac13 e(H).
 $$
 
 On the other hand $\tau_3(H)\le \frac12\, e(H)$: take a bipartite subgraph of $H$ with
at least half of the edges; then the remaining edges form a $K_3$-cover. Putting
the last two inequalities together, we obtain the required result.
\end{proof}

Let $1\le i\le s$. We now describe an algorithm for finding a $K_3$-cover $F_i$ in $G_i$. Initially, let $H=G_i$ and $F_i=\emptyset$. Repeat the following. 

Take a minimum fractional $K_3$-cover $f$ of $H$.
If the first alternative of Lemma~\ref{lm:K} is true, pick a $K_3$-cover of $H$ of size
at most $\frac32\, \tau_3^*(H)$, add it to $F_i$ and stop. Otherwise, fix some edge $xy\in E(H)$
returned by Lemma~\ref{lm:K}. Let $F'$ consist of all pairs $xz$ and $yz$ over $z\in \Gamma(x)\cap \Gamma(y)$. Add $F'$ to $F_i$ and remove $F'$ from $E(H)$. Repeat the whole step (with the
new $H$ and $f$).

Consider any moment during this algorithm, when we had $f(xy)=0$ for some edge $xy$ of $H$.
Since $f$ is a fractional $K_3$-cover, we have that $f(xz)+f(yz)\ge 1$ for every $z\in \Gamma(x)\cap \Gamma(y)$. Thus, if $H'$ is obtained from $H$ by removing $2\ell$ such pairs, where $\ell=|\Gamma(x)\cap \Gamma(y)|$, then $\tau_3^*(H')\le \tau_3^*(H)-\ell$ because $f$ when restricted to $E(H')$ is
still a fractional cover (although not necessarily an optimal one). Clearly, $|F_i|$ increases 
by $2\ell$ during this operation. Thus, indeed we obtain, at the end, a $K_3$-cover $F_i$ of $G_i$
of size at most $2\tau_3^*(G_i)$. 

Also, by (\ref{eq:duality}) and (\ref{eq:HR}) we have that 
$$\sum_{i=1}^s |F_i|\le 2 \sum_{i=1}^s \nu_3(G_i) +o(n^2).$$

Now, since all slacks in (\ref{eq:main}) are $o(n^2)$, we conclude that
$$\sum_{i=1}^s\nu_3(G_i)\le \frac m2+o(n^2)$$ and (\ref{eq:sumFi}) holds. In fact,  (\ref{eq:sumFi}) is
equality by (\ref{eq:tauGEq}).

Recall that $G_i'$ is obtained from $G_i$ by removing all edges of $F_i$ and $G'$ is
the edge-disjoint union of the graphs $G_i'$.  Suppose that there exist $V_1,\dots,V_{R-1}$ 
satisfying (\ref{eq:diff}).
Let $M=E(T)\setminus E(G')$ consist of \emph{missing} edges. Thus, $|M|=o(n^2)$.

Let 
 $$X=\{x\in V(T) \mid \deg_M(x)\ge c_2 n\},
  $$ 
  where we define $c_2=(4(R-1))^{-1}$. Clearly, 
$$|X|\le 2|M|/c_2n=o(n).$$

Observe that, for every $1\le i\le s$, if the first alternative of Lemma~\ref{lm:K} holds at some point,
then the remaining graph $H$ satisfies $\tau_3^*(H)=o(n^2)$. Indeed,  otherwise by $\tau_3(G_i)\le 2\tau_3^*(G_i)-\tau^*_3(H)/2+o(n^2)$ we get a strictly
smaller constant than $2$ in (\ref{eq:tauK3}) and thus a gap of $\Omega(n^2)$ in (\ref{eq:main}),
a contradiction. Therefore, all but $o(n^2)$ edges in $F_i$ come from some \emph{parent edge} $xy$ that had $f$-weight 0 at some point.

When our algorithm adds pairs
$xz$ and $yz$ to $F_i$ with the same parent $xy$, then it adds the same number of pairs incident to $x$ as those incident to
$y$. Let $\mathcal P$ consist
of pairs $xy$ that are disjoint from $X$ and were a parent edge during the run of the algorithm.
Since the total number of pairs in $F_i$ incident to $X$ is at most $n|X|=o(n^2)$, there are
$|F_i|-o(n^2)$ pairs in $F_i$ such that their parent is in $\mathcal P$.

Let us show that $y_0$ and $y_1$ belong to different parts $V_j$ for every pair $y_0y_1\in \mathcal P$. Suppose on the contrary that, say, $y_0,y_1\in V_1$.
 For each $2\le j\le R-1$ pick an arbitrary
$y_j\in V_j\setminus (\Gamma_M(y_0)\cup \Gamma_M(y_1))$. Since $y_0,y_1\not\in X$, the possible 
number of choices for $y_j$ is at least 
$$\frac n{R-1}-2c_2n+o(n)\ge \frac n{R-1}-3c_2n. $$

 Let 
 $$
  Y=\{y_0,\dots,y_{R-1}\}.
   $$
   
   By the above, we have at least $(\frac n{R-1}-3c_2n)^{R-2}=\Omega(n^{R-2})$ choices of $Y$. Note that by the definition,
all edges between $\{y_0,y_1\}$ and the rest of $Y$ are present in $E(G')$.   
Thus, the number of sets $Y$ containing at least one edge of $M$ different from
$y_0y_1$ is at most
 $$
  |M| \times n^{R-4}=o(n^{R-2}).
   $$
  This is $o(1)$ times the number of choices of $Y$. Thus, for almost every $Y$, $H=G'[Y]$ is a clique
(except perhaps the pair $y_0y_1$). In particular, there is at least one such choice of $Y$;
fix it. Let $i\in \{1,\ldots, k\}$ be arbitrary. Adding back the pair $y_0y_1$ coloured $i$ to  $H$ (if it is not there already),
we obtain a $k$-edge-colouring of the complete graph $H$ of order $R$. By the definition of 
$R=R(r_1,\dots,r_k)$,
there must be a monochromatic triangle on $abc$ of colour $h\le s$. (Recall that we assumed at the
beginning that $G_j$ is $K_{r_j}$-free for each $j>s$.) But $abc$ has to contain an edge from 
the $K_3$-cover $F_h$, say $ab$. This edge $ab$ is not in $G'$ (it was removed from $G$).
If $a,b$ lie in different parts $V_j$, then $ab\in M$, a contradiction to the choice of $Y$.
The only possibility is that $ab=y_0y_1$. Then $h=i$. Since both $y_0c$ and $y_1c$ are in $G_i'$,
they were never added to the $K_3$-cover $F_i$ by our algorithm. Therefore, $y_0y_1$
was never a parent, which is the desired contradiction.

Thus, every  $xy\in\mathcal P$ connects two different parts $V_j$. For every such
parent $xy$, the number
of its children in $M$ is at least half of all its children. Indeed, for every pair of children $xz$ and $yz$, at least one connects two different parts; this child necessarily belongs to $M$. Thus,
 \begin{equation*}\label{eq:aim}
|F_i\cap M|\ge \frac12\,|F_i|+o(n^2).
\end{equation*}
 (Recall that parent edges that intersect $X$ produce at most $2n|X|=o(n^2)$ children.) Therefore,
 $$
  |M|\ge \frac12\, \sum_{i=1}^s |F_i|+o(n^2)\ge \frac m2+o(n^2)=\Omega(n^2),
   $$ 
    contradicting
 (\ref{eq:diff}).  This contradiction proves Lemma~\ref{lm:new}.
 \end{proof}

We are now able to prove Theorem \ref{Krexact}.

\begin{proof}[Proof of the upper bound in Theorem \ref{Krexact}]
Let $C$ be the constant returned by Theorem \ref{EGthm} for $r=R$.  Let $n_0=n_0(r_1,\dots,r_k)$ be sufficiently large to satisfy all the inequalities we will encounter. Let  $G$ be a $k$-edge-coloured graph on $n\geq n_0$ vertices. We will show that $\phi_k(G,\mathcal C) \leq t_{R-1}(n)$ with equality if and only if $G =T_{R-1}(n)$, and $G$ does not contain a monochromatic copy of $K_{r_i}$ in colour $i$ for every $1\le i\le k$.
 
Let $e(G)=t_{R-1}(n)+m$, where $m$ is an integer. If $m< 0$, we can decompose $G$  into single edges and there is nothing to prove. 

Suppose $m=0$. If $G$ contains a monochromatic copy of $K_{r_i}$ in colour $i$ for some $1\le i\le k$, then $G$ admits a monochromatic $\mathcal C$-decomposition with at most $t_{R-1}(n)- {r_i \choose 2 }+1<t_{R-1}(n)$ parts and we are done. Otherwise, the definition of $R$ implies that $G$ does not contain a copy of $K_{R}$. Therefore,
 $G=T_{R-1}(n)$ by Tur\'{a}n's theorem and $\phi_k(G,\mathcal C)=t_{R-1}(n)$ as required.

Now suppose $m>0$. We can also assume that $m< {n\choose 2}/C$ for otherwise we are done: $\phi_k(G,\mathcal C)<t_{R-1}(n)$ by Lemma \ref{lm:new}. Thus, by Theorem \ref{EGthm}, the graph $G$ contains at least $m-Cm^2/n^2>\frac{m}{2}$ edge-disjoint copies of $K_{R}$. Since each $K_{R}$ contains a monochromatic copy of $K_{r_i}$ in the colour-$i$ graph $G_i$, for some $1\le i\le k$, we conclude that $\sum_{i=1}^k \nu_{r_i}(G_i) >\frac{m}{2}$, so that $\sum_{i=1}^k({r_i\choose 2}-1)\nu_{r_i}(G_i)\ge \sum_{i=1}^k 2\nu_{r_i}(G_i)>m$. We have
\begin{equation*}
\phi_k(G,\mathcal C) = e(G)-\sum_{i=1}^k{r_i \choose 2 }\nu_{r_i}(G_i) + \sum_{i=1}^k\nu_{r_i}(G_i)< t_{R-1}(n),
\end{equation*}
giving the required. 
\end{proof}

\noindent\textbf{Remark.} By analysing the above argument, one can also derive the following stability property for every fixed family $\mathcal C$ of cliques as $n\to\infty$: every graph $G$ on $n$ vertices with $\phi_k(G,\mathcal{C})=t_{R-1}(n)+o(n^2)$ is $o(n^2)$-close to the Tur\'an graph $T_{R-1}(n)$ in the edit distance.%

\section*{Acknowledgements}

Henry Liu and Teresa Sousa acknowledge the support from FCT - Funda\c c\~ao para a Ci\^encia e a Tecnologia (Portugal), through the projects PTDC/MAT/113207/2009 and PEst-OE/MAT/UI0297/2011 (CMA). 
Oleg Pikhurko was supported by ERC grant~306493 and EPSRC grant~EP/K012045/1.

The authors thank the anonymous referees for the careful reading of the manuscript.

\bibliography{mono-decompositions-bib}
\bibliographystyle{abbrv}

\end{document}